\documentclass[11pt]{article}

\pdftrailerid{}
\usepackage[T1]{fontenc}
\usepackage{lmodern}
\usepackage{amsmath,amssymb,amsthm,mathtools,mathrsfs}
\usepackage[margin=1.02in]{geometry}
\usepackage{booktabs,tabularx,array,longtable}
\usepackage{microtype}
\usepackage{xcolor}
\usepackage{enumitem}
\usepackage[colorlinks=true,linkcolor=blue!50!black,
  citecolor=blue!50!black,urlcolor=blue!50!black]{hyperref}

\hypersetup{
  pdftitle={The second pole of Witten zeta functions and exact evaluations in types F4 and D5},
  pdfauthor={Jonas Matuzas},
  pdfsubject={Witten zeta functions, the first pole below the leading pole, projective chamber periods, and exact F4 and D5 residues},
  pdfkeywords={Witten zeta function, second pole, root system, projective chamber period, Varchenko matrix, Dixon sum, Selberg integral, F4, D5}}

\allowdisplaybreaks
\setlist[itemize]{leftmargin=1.5em,itemsep=2pt,topsep=3pt}
\setlist[enumerate]{leftmargin=1.7em,itemsep=2pt,topsep=3pt}

\newtheorem{theorem}{Theorem}[section]
\newtheorem{proposition}[theorem]{Proposition}
\newtheorem{corollary}[theorem]{Corollary}
\newtheorem{lemma}[theorem]{Lemma}
\newtheorem{remark}[theorem]{Remark}

\newcommand{\R}{\mathbb{R}}

\newcommand{\Q}{\mathbb{Q}}

\newcommand{\Res}{\operatorname*{Res}}

\newcommand{\Disc}{\operatorname{Disc}}
\newcommand{\dd}{\,\mathrm{d}}
\newcommand{\Per}{\mathcal{P}}

\newcommand{\cD}{\mathscr{D}}

\newcommand{\eps}{\varepsilon}

\newcommand{\remd}{\operatorname{rem}}

\title{The second pole of Witten zeta functions\\
and exact evaluations in types \(F_4\) and \(D_5\)}
\author{Jonas Matuzas\\
\href{mailto:jonas.matuzas@gmail.com}{\nolinkurl{jonas.matuzas@gmail.com}}}
\date{}

\begin{document}
\maketitle

\begin{abstract}
Let \(\Phi\) be an irreducible reduced crystallographic root system of rank
\(r\ge2\), let \(N=|\Phi^+|\), and let \(h\) be its Coxeter number.  For the
normalized Witten zeta function \(\xi_\Phi\), we determine the first distinct
pole below the leading pole \(2/h\).  It occurs at
\[
 q_2(\Phi)=\frac{r-1}{N-1}=\frac{2(r-1)}{rh-2}.
\]
The pole is simple.  Its only contributions come from the codimension-one
simple-root faces, and
\[
 \operatorname*{Res}_{s=q_2}\xi_\Phi(s)
 =\frac{\zeta_{\mathrm R}(q_2)}{N-1}
  \sum_{i=1}^{r}\mathcal P_i(q_2)<0,
\]
where each wall period \(\mathcal P_i(q_2)\) is finite and positive.

We also prove a Stokes relation for projective arrangement periods.  If the
weights satisfy the critical homogeneity condition and the usual strict
flat-integrability inequalities, then the vector of positive chamber periods
is a null vector of the associated Varchenko matrix.  This relation makes it
possible to evaluate wall periods that are not accessible from a single
chamber.  In type \(F_4\), a two-orbit chamber relation and Dixon's
\({}_3F_2(1)\) summation give a pure gamma-product formula for the residue at
\(3/23\).  In type \(D_5\), a four-orbit relation and Selberg's integral give
a pure gamma-product formula for the complete wall sum and the residue at
\(4/19\).
\end{abstract}



\begingroup
\emergencystretch=1em
\medskip\noindent\textbf{Generative-AI disclosure and author responsibility.}\par
\noindent This work was produced using OpenAI's ChatGPT 5.6 Pro. The author directed
and audited the work throughout. Jonas Matuzas takes full responsibility for
the mathematics and the final text.
\endgroup

\section{Introduction}
\label{sec:introduction}

For a complex simple Lie algebra, or equivalently its simply connected compact
form, the Witten zeta function is
\begin{equation}\label{eq:witten-def-intro}
 \zeta_\Phi(s)=\sum_{\lambda\in P_+}(\dim V_\lambda)^{-s},
\end{equation}
where \(P_+\) is the set of dominant integral weights.  These functions arose
in two-dimensional gauge theory \cite{Witten1991} and are also representation
zeta functions.  Their abscissa of convergence is \(2/h\)
\cite{LarsenLubotzky2008,HasaStasinski2019}, and the residue at that leading
pole has a uniform Macdonald--Mehta evaluation \cite{MatuzasLeading2026}.

The analytic continuation of polynomial Dirichlet series and its candidate
singular geometry go back to Lichtin, Essouabri, and Peter
\cite{Lichtin1988,Essouabri1997,Peter1998}.  In the root-system setting,
Komori, Matsumoto, and Tsumura developed the multivariable theory and described
its possible singular hyperplanes
\cite{MatsumotoTsumura2006,KMT2010,KMT2011}
\cite[Chapter~7, Theorem~7.8]{KMTBook2023}.  For Euler--Zagier multiple zeta
functions, Zhao computed residues in arbitrary depth, and Akiyama, Egami, and
Tanigawa determined which candidate hyperplanes survive
\cite{Zhao2000,AkiyamaEgamiTanigawa2001}.  The companion paper
\cite{MatuzasGenericPoles2026} gives the corresponding generic divisor and
flag-residue formulas for root-system zeta functions.  Budur, Shi, and Zuo
provide a close recent comparison in the continuous setting of multivariable
Archimedean zeta integrals \cite{BudurShiZuo2025}; their local-integral problem
is different from the discrete lattice series considered here.

The first theorem of this paper identifies the first distinct pole below
\(2/h\) in every irreducible type.  If \(S\) is a proper set of simple nodes,
the unshifted face associated with \(S\) has candidate value
\[
 q(S)=\frac{|S|}{N(S)},
\]
where \(N(S)\) is the number of positive coroots whose support meets \(S\).
A strict comparison of parabolic root counts shows that the maximum of
\(q(S)\) is attained exactly when the complement of \(S\) is a single node.
This gives
\[
 q_2(\Phi)=\frac{r-1}{N-1}.
\]
The same strict inequalities imply convergence of the wall periods.  Since
the maximizing supports are pairwise incomparable and no full-support shifted
candidate occurs, no nested boundary flag can raise the pole order.  The
periods are positive and \(\zeta_{\mathrm R}(q_2)<0\), so the residue is
nonzero and negative.

The second theorem concerns real hyperplane arrangements.  Let \(P_D\) be the
positive projective period of a chamber \(D\), and let
\(V(a)\) be the Varchenko matrix with weights
\(a_H=e^{\pi i\lambda_H}\).  Under the critical equality
\(\sum_H\lambda_H=\dim V\) and strict inequalities on every proper flat, we
prove
\[
 V(a)P=0.
\]
Varchenko's determinant formula gives the degeneracy locus of this matrix, and
Falk--Varchenko construct the projective contravariant form at zero total
additive weight \cite{Varchenko1993,FalkVarchenko2012}.  The point needed here
is more specific: the positive real chamber-period vector itself is an
explicit null vector.  The proof is a direct Stokes argument on translated
spheres, with the boundary phases computed chamber by chamber.

The arrangement relation is then applied twice.  For \(F_4\), the relevant
rank-three restriction has two chamber orbits.  One orbit reduces to a
well-poised \({}_3F_2(1)\), evaluated by Dixon's formula, and the null-vector
relation transfers that value to the second orbit.  For \(D_5\), the
rank-four restriction has four chamber orbits.  One orbit is a Selberg
integral, and the same relation determines all five marked wall periods.
These calculations give exact gamma-product formulas for the normalized and
ordinary Witten-zeta residues.

Detailed single-variable calculations in ranks two and three are due to Romik
and Au \cite{Romik2017,Au2024}.  The \(A_2\), \(B_2\), and \(G_2\)
computations below are included only as normalization checks: they recover
Au's formulas exactly and are not used as evidence for the general theorem.
The new content is the uniform all-type result, the projective chamber-period
relation, and the higher-rank \(F_4\) and \(D_5\) evaluations.  No prior
uniform theorem or these two exact evaluations were found in the searched
literature; this statement is a description of that search, not an absolute
priority claim.

Section~\ref{sec:normalization} fixes the normalization and derives the candidate poles and simple-face residue formula used in the sequel.  Section~\ref{sec:universal-pole} proves the universal theorem, and Section~\ref{sec:marked-walls} records the geometric normalization of the wall periods.  Section~\ref{sec:chamber-kernel} proves the chamber-period relation.  Sections~\ref{sec:F4-period} and \ref{sec:D5-closed} contain the exact higher-rank evaluations.  Appendices~\ref{app:f4-certificate} and \ref{app:d5-certificate} record the finite cyclotomic calculations, while Appendix~\ref{sec:low-rank} gives low-rank normalization checks.

\paragraph{Scope.}
The paper does not claim a gamma-product formula for every wall period.  It
makes no transcendence or algebraic-independence assertion about the gamma
values or about \(\zeta_{\mathrm R}(q_2)\).  For reducible semisimple Lie
algebras, pole orders and Laurent coefficients are obtained by multiplying the
simple-factor zeta functions; there is no formula depending only on total rank
and total root count.

\section{Normalization and candidate poles}
\label{sec:normalization}

Let \(\Psi=\Phi^\vee\), choose simple coroots
\(\beta_1,\ldots,\beta_r\), and write
\[
 \beta=\sum_{i=1}^r b_i(\beta)\beta_i,
 \qquad b_i(\beta)\in\mathbb Z_{\ge0}
\]
for every positive coroot \(\beta\).  If
\(\lambda=\sum_i n_i\omega_i\) is a dominant highest weight, put
\(m_i=n_i+1\).  Weyl's dimension formula gives
\[
 \dim V_\lambda=\frac{P_\Phi(m)}{K_\Phi},
 \qquad
 P_\Phi(m)=\prod_{\beta\in\Psi^+}
 \left(\sum_{i=1}^r b_i(\beta)m_i\right),
\]
where
\[
 K_\Phi=\prod_{\beta\in\Psi^+}\langle\rho,\beta\rangle
 =\prod_{j=1}^r e_j!
\]
and \(e_1,\ldots,e_r\) are the exponents of \(\Phi\).  We use the normalized
function
\begin{equation}\label{eq:xi-def}
 \xi_\Phi(s)=K_\Phi^{-s}\zeta_\Phi(s)
 =\sum_{m\in\mathbb N^r}P_\Phi(m)^{-s}.
\end{equation}
Thus normalization does not change pole locations or orders.  If \(q\) is a
pole, then
\[
 \operatorname*{Res}_{s=q}\zeta_\Phi(s)
 =K_\Phi^q\operatorname*{Res}_{s=q}\xi_\Phi(s).
\]

Put \(I=\{1,\ldots,r\}\) and \(N=|\Psi^+|\).  For a nonempty subset
\(S\subseteq I\), define
\[
 R(S)=\{\beta\in\Psi^+:\operatorname{supp}(\beta)\cap S\ne\varnothing\},
 \qquad N(S)=|R(S)|.
\]
For a proper support \(S\), let
\[
 \Delta_S=\left\{u_i>0:\sum_{i\in S}u_i=1\right\},
\]
\[
 Q_S(u)=\prod_{\beta\in R(S)}
 \left(\sum_{i\in S}b_i(\beta)u_i\right),
 \qquad
 \mathcal P_S(q)=\int_{\Delta_S}Q_S(u)^{-q}\,du.
\]
The simplex carries the affine lattice measure obtained by eliminating one
coordinate.

\begin{proposition}[Candidate poles and simple-face residues]\label{prop:carrier-input}
Every diagonal pole belongs to
\begin{equation}\label{eq:diagonal-candidates}
 \left\{\frac{|S|-\ell}{N(S)}:
 \varnothing\ne S\subsetneq I,\ \ell\in\mathbb Z_{\ge0}\right\}
 \cup\left\{\frac rN\right\}.
\end{equation}
The full support contributes only \(r/N\).  If a proper support \(S\) is
regular at its unshifted value \(q=|S|/N(S)\), then its contribution to the
diagonal residue is
\begin{equation}\label{eq:carrier-residue}
 \frac1{N(S)}\mathcal P_S(q)\,
 \xi_{\Psi_{I\setminus S}}(q).
\end{equation}
A pole of order greater than one can occur only when the incident supports
contain a strict inclusion chain.
\end{proposition}

\begin{proof}
We recall the part of the multivariable argument needed here.  Introduce one
exponent \(s_\beta\) for every positive coroot and consider
\[
 Z_\Phi(\mathbf s)
 =\sum_{m\in\mathbb N^r}
  \prod_{\beta\in\Psi^+}
  \left(\sum_i b_i(\beta)m_i\right)^{-s_\beta}.
\]
Fix a proper support \(S\) and split the summation variables into their
\(S\)- and \(S^c\)-coordinates.  For a root meeting both parts, write the
corresponding linear form as \(A+B\).  Mellin--Barnes expansion gives
\[
 (A+B)^{-s}
 =\frac1{2\pi i}\int_{(c)}
   \frac{\Gamma(s+z)\Gamma(-z)}{\Gamma(s)}A^{-s-z}B^z\,dz,
 \qquad -\Re s<c<0.
\]
Moving the contours to the right produces residues indexed by nonnegative
integers.  If the sum of those integers is \(\ell\), the block in the
\(S\)-variables is homogeneous of degree
\[
 -\sum_{\beta\in R(S)}s_\beta-\ell.
\]
Its radial integral can therefore have a pole only when
\[
 \sum_{\beta\in R(S)}s_\beta=|S|-\ell.
\]
On the diagonal this is the first set in \eqref{eq:diagonal-candidates}.  At
\(\ell=0\), the residue separates into the projective simplex integral and
the complementary root-system zeta function.  The derivative of the carrier
form along the diagonal is \(N(S)\), which gives
\eqref{eq:carrier-residue}.

For full support there is no complementary block.  A smooth partition of the
closed simplex separates the interior from neighborhoods of its proper faces.
The interior amplitude is exactly homogeneous, so Poisson summation leaves a
single radial denominator
\(\sum_{\beta\in\Psi^+}s_\beta-r\).  The remaining angular pieces are the
proper-support sectors already considered.  Hence no shifted full-support
candidate occurs.

Finally, resolve the coordinate faces in increasing order of dimension.  In
the resulting normal-crossings charts, products of boundary denominators are
indexed by nested collections of supports; for the maximal support building
set these are strict chains.  Thus a higher-order diagonal pole requires a
strict chain of incident supports.  Complete details, including the
multivariable residue functions, are given in
\cite{MatuzasGenericPoles2026}.
\end{proof}

\section{The universal second pole}
\label{sec:universal-pole}

The proof of the main theorem rests on one strict root-count inequality.

\begin{lemma}[Strict parabolic density]\label{lem:parabolic-density}
Let \(\Psi\) be irreducible of rank \(r\), with \(N\) positive roots.  If
\(\Psi_J\) is a proper standard parabolic subsystem of rank \(j>1\), with
\(N_J\) positive roots, then
\begin{equation}\label{eq:strict-density}
 \boxed{
 \frac{N_J-1}{j-1}<\frac{N-1}{r-1}.}
\end{equation}
Equivalently,
\[
 \frac{r-j}{N-N_J}<\frac{r-1}{N-1}.
\]
\end{lemma}

\begin{proof}
Duality only exchanges \(B_r\) and \(C_r\), so it does not affect positive
root counts.

For \(A_r\), one has \(N=r(r+1)/2\).  Convexity shows that a rank-\(j\)
standard parabolic has at most \(j(j+1)/2\) positive roots.  Hence
\[
 \frac{N_J-1}{j-1}\leq\frac{j+2}{2}<\frac{r+2}{2}
 =\frac{N-1}{r-1}.
\]
For \(B_r\) and \(C_r\), \(N=r^2\) and \(N_J\leq j^2\), giving
\[
 \frac{N_J-1}{j-1}\leq j+1<r+1=\frac{N-1}{r-1}.
\]
For \(D_r\), \(N=r(r-1)\).  At \(j=2\), the largest parabolic is \(A_2\)
and has three positive roots.  For \(j\geq3\), the largest is \(D_j\)
(with \(D_3=A_3\)), so \(N_J\leq j(j-1)\).  Thus
\[
 \frac{N_J-1}{j-1}\leq j-\frac1{j-1}
 <r-\frac1{r-1}=\frac{N-1}{r-1}.
\]
For the exceptional types, the maximal positive-root count at every proper
rank is shown below; disconnected subdiagrams never exceed the displayed
connected maximum.
\begin{center}
\begin{tabular}{@{}ccl@{}}
\toprule
ambient type & \(N\) & maximal proper standard parabolics by rank \(j\geq2\)\\
\midrule
\(G_2\)&6&none\\
\(F_4\)&24&\(B_2(4),\ B_3/C_3(9)\)\\
\(E_6\)&36&\(A_2(3),\ A_3(6),\ D_4(12),\ D_5(20)\)\\
\(E_7\)&63&\(A_2(3),\ A_3(6),\ D_4(12),\ D_5(20),\ E_6(36)\)\\
\(E_8\)&120&\(A_2(3),\ A_3(6),\ D_4(12),\ D_5(20),\ E_6(36),\ E_7(63)\)\\
\bottomrule
\end{tabular}
\end{center}
Direct substitution is strict in every row.  This completes the finite
classification proof.
\end{proof}

For the wall opposite node \(i\), put
\[
 S_i=I\setminus\{i\},
 \qquad
 \Delta_i=\left\{u_j>0:\sum_{j\ne i}u_j=1\right\},
\]
and
\begin{equation}\label{eq:Qi-def}
 Q_i(u)=\prod_{\beta\in\Psi^+\setminus\{\beta_i\}}
 \left(\sum_{j\ne i}b_j(\beta)u_j\right),
 \qquad
 \Per_i(q)=\int_{\Delta_i}Q_i(u)^{-q}\dd u.
\end{equation}

\begin{lemma}[Convergence at the second-pole exponent]
\label{lem:wall-convergence}
Let
\[
 q_2=\frac{r-1}{N-1}.
\]
Then \(\Per_i(q_2)\) is finite and positive for every simple node \(i\).
\end{lemma}

\begin{proof}
Positivity is immediate because every factor of \(Q_i\) is positive in the
interior of \(\Delta_i\).  It remains to check the boundary of the simplex.
Let \(A\) be a nonempty proper subset of \(S_i\), and put
\[
 J=A\cup\{i\}.
\]
Along the relative interior of the face \(u_j=0\) for \(j\in A\), a
restricted root factor vanishes precisely when the original positive root is
supported on \(J\), except for the omitted root \(\beta_i\).  Thus exactly
\(N_J-1\) factors vanish, where \(N_J\) is the number of positive roots in
the standard parabolic subsystem on \(J\).

In the real blow-up chart for this face, write \(u_j=t v_j\) for \(j\in A\).
The simplex measure contributes \(t^{|A|-1}\,dt\), while the integrand
contributes \(t^{-q_2(N_J-1)}\).  Local integrability is therefore equivalent
to
\begin{equation}\label{eq:wall-integrability-criterion}
 q_2(N_J-1)<|A|.
\end{equation}
Since \(|J|=|A|+1>1\) and \(J\subsetneq I\), Lemma~\ref{lem:parabolic-density}
gives
\[
 \frac{N_J-1}{|A|}<\frac{N-1}{r-1}=\frac1{q_2},
\]
which is exactly \eqref{eq:wall-integrability-criterion}.  The same argument
applies to every exceptional divisor in the iterated blow-up of the simplex,
so the pullback of \(Q_i^{-q_2}\dd u\) is locally integrable at every boundary
stratum.  Hence \(\Per_i(q_2)\) converges.
\end{proof}

\begin{theorem}[Universal second pole]\label{thm:universal-second-pole}
Let \(\Phi\) be irreducible, reduced, crystallographic, and of rank
\(r\geq2\).  Then the next distinct pole of \(\xi_\Phi\) to the left of the
leading pole \(q_1=2/h\) is
\begin{equation}\label{eq:q2-main}
 \boxed{
 q_2(\Phi)=\frac{r-1}{N-1}=\frac{2(r-1)}{rh-2}.}
\end{equation}
Its only contributing supports are \(S_1,\ldots,S_r\); the pole is simple; and
\begin{equation}\label{eq:universal-residue}
 \boxed{
 \Res_{s=q_2}\xi_\Phi(s)
 =\frac{\zeta_{\mathrm R}(q_2)}{N-1}
 \sum_{i=1}^r\Per_i(q_2).}
\end{equation}
Every \(\Per_i(q_2)\) is finite and positive.  Hence the residue is strictly
negative and nonzero.  For the ordinary Witten zeta function,
\[
 \Res_{s=q_2}\zeta_\Phi(s)
 =K_\Phi^{q_2}\Res_{s=q_2}\xi_\Phi(s).
\]
\end{theorem}

\begin{proof}
For fixed \(S\), every shifted candidate is strictly below the unshifted one,
so it is enough to maximize \(q(S)=|S|/N(S)\) over nonempty proper supports.
Let \(J=I\setminus S\), \(j=|J|\), and let \(N_J\) be the number of positive
roots in the standard parabolic on \(J\).  A positive root fails to meet
\(S\) precisely when it is supported on \(J\), so
\[
 N(S)=N-N_J.
\]
If \(j=1\), then \(N_J=1\) and
\[
 q(S)=\frac{r-1}{N-1}.
\]
If \(j>1\), Lemma~\ref{lem:parabolic-density} gives a strict inequality.
Thus the maximizing supports are exactly the complements of single nodes.
The full-support shifted family is absent, and
\[
 q_1-q_2=\frac{N-r}{N(N-1)}>0,
\]
so there is no collision with the leading pole.

Lemma~\ref{lem:wall-convergence} proves that every wall period
\(\Per_i(q_2)\) is finite and positive.  The only positive root supported on
the omitted node is \(\beta_i\); hence the complementary factor in
Proposition~\ref{prop:carrier-input} is
\(\xi_{A_1}(q_2)=\zeta_{\mathrm R}(q_2)\), and the derivative of the defining
linear form along the diagonal is \(N-1\).  Summing the simple-wall
contributions gives \eqref{eq:universal-residue}.

The supports \(S_i\) are pairwise incomparable.  The only support properly
containing any \(S_i\) is the full support, whose sole candidate is
\(2/h\ne q_2\).  Thus no strict incident chain occurs at \(q_2\), and the
pole is at most simple.  Finally \(0<q_2<1\), and
\[
 \zeta_{\mathrm R}(q)=\frac{\eta(q)}{1-2^{1-q}}<0
 \qquad(0<q<1),
\]
where \(\eta(q)>0\) is the alternating eta function.  All wall periods are
positive, so no cancellation is possible.  The pole is exactly simple and
its residue is negative.
\end{proof}

\begin{remark}[Rank one]
For \(A_1\), \(\xi_{A_1}(s)=\zeta_{\mathrm R}(s)\); there is only the pole at
\(s=1\), so no second pole exists.
\end{remark}

\begin{corollary}[Closed formulas for the location]\label{cor:q2-types}
The universal formula specializes to
\begingroup
\renewcommand{\arraystretch}{1.35}
\[
\begin{array}{c|c@{\qquad}c|c}
\Phi&q_2(\Phi)&\Phi&q_2(\Phi)\\ \hline
A_r&\dfrac{2}{r+2}&G_2&\dfrac15\\
B_r,C_r&\dfrac1{r+1}&F_4&\dfrac3{23}\\
D_r&\dfrac{r-1}{r(r-1)-1}&E_6&\dfrac17\\
&&E_7&\dfrac3{31}\\
&&E_8&\dfrac1{17}
\end{array}
\]
\endgroup
\end{corollary}

\section{Geometric forms of the wall periods}
\label{sec:marked-walls}

The integral \(\Per_i\) is not determined by the abstract weighted central
arrangement alone.  Its marking includes the selected positive chamber,
primitive lattice covectors and their scalar factors, the affine section
\(\sum u_j=1\), and its lattice-normalized measure.

The elementary \(A_3\) case already detects this distinction.  At
\(q_2=2/5\), the endpoint and central wall polynomials can be written as
\[
 Q_E(x,y)=x^2y(x+y)^2,
 \qquad
 Q_C(x,y)=x^2y^2(x+y).
\]
The underlying unmarked weighted arrangements are linearly equivalent, but
the marked periods are
\[
 \Per_E=B\!\left(\frac15,\frac35\right),
 \qquad
 \Per_C=B\!\left(\frac15,\frac15\right),
\]
and
\[
 \frac{\Per_E}{\Per_C}
 =\frac{\sin(\pi/5)}{\sin(2\pi/5)}
 =\frac{\sqrt5-1}{2}.
\]
There are two endpoint walls and one central wall, so
\begin{equation}\label{eq:A3-aggregate}
 \boxed{
 2\Per_E+\Per_C
 =\sqrt5\,B\!\left(\frac15,\frac15\right).}
\end{equation}

We now complete a marked wall over all chambers in its Weyl orbit.  Choose a
coroot \(\alpha\), let \(H_\alpha=\alpha^\perp\), and let \(Q_\alpha\) be
the canonical restricted discriminant with its primitive lattice scalars and
induced projective measure retained.  For the coroot orbit \(\mathcal O\) of
\(\alpha\), define
\begin{equation}\label{eq:J-orbit-def}
 J_{\mathcal O}(q)=
 \sum_{D\in\operatorname{Ch}(\mathcal A^{H_\alpha})}
 \int_{\mathbb P(D)}|Q_\alpha|^{-q}\,\Omega_\alpha.
\end{equation}
By Lemma~\ref{lem:affine-projective-spherical}, all linear chambers are counted with the same lattice normalization; antipodal chambers are distinct.  Hence
\[
 J_{\mathcal O}(q)=
 \int_{S(H_\alpha)}|Q_\alpha(\theta)|^{-q}\dd\sigma(\theta).
\]

\begin{lemma}[Affine--projective--spherical normalization]
\label{lem:affine-projective-spherical}
Let \(V\) be an oriented Euclidean space of dimension \(d\), let \(L\subset V\)
be a full lattice of Euclidean covolume \(\delta_L\), let \(D\) be a chamber of
a homogeneous polynomial \(Q\) of degree \(p\), and put \(q=d/p\).  If a
linear form \(a\) is positive on \(D\) and the section
\(\Sigma_a=D\cap\{a=1\}\) is compact, then
\begin{equation}\label{eq:affine-spherical-normalization}
 \frac1{\delta_L}
 \int_{\Sigma_a}|Q(x)|^{-q}\,\iota_E(dx_1\wedge\cdots\wedge dx_d)
 =\frac1{\delta_L}
 \int_{D\cap S(V)}|Q(\theta)|^{-q}\dd\sigma(\theta).
\end{equation}
Here \(E=\sum_jx_j\partial_{x_j}\), and the restrictions of the contracted
form carry the boundary orientation induced by the radial coordinate.
\end{lemma}

\begin{proof}
The \((d-1)\)-form
\[
 \omega_Q=|Q|^{-d/p}\iota_E(dx_1\wedge\cdots\wedge dx_d)
\]
is invariant under every positive radial dilation.  Each ray in \(D\) meets
both \(\Sigma_a\) and \(D\cap S(V)\) exactly once.  The radial identification
therefore pulls the restriction of \(\omega_Q\) on one section to its
restriction on the other.  Replacing Euclidean volume by the lattice-normalized
volume form multiplies both sides by \(\delta_L^{-1}\).
\end{proof}

\begin{proposition}[Spherical forms of the leading and wall periods]
\label{prop:inverse-jacobian-moments}
Let \(V_\Phi=\operatorname{span}_{\R}\Psi\) carry the invariant Euclidean
metric, and define the simple-coroot evaluation map
\[
 A_\Phi:V_\Phi\longrightarrow\R^r,
 \qquad
 A_\Phi(x)=\bigl((x,\beta_1),\ldots,(x,\beta_r)\bigr),
 \qquad
 \delta_\Phi=|\det A_\Phi|.
\]
If \(G_\Psi=((\beta_j,\beta_k))_{j,k}\) is the simple-coroot Gram
matrix, then
\begin{equation}\label{eq:bulk-coordinate-determinant}
 \boxed{\delta_\Phi^2=\det G_\Psi.}
\end{equation}
Put
\[
 \mathcal S_\Phi
 =\{\theta\in S(V_\Phi):(\theta,\beta_j)>0\ \text{for all }j\},
\]
\[
 J_\Phi(x)=\prod_{\beta\in\Psi^+}(x,\beta),
 \qquad
 D_\Phi(x)=J_\Phi(x)^2.
\]
Then the normalized leading residue has the spherical moment form
\begin{equation}\label{eq:bulk-inverse-jacobian}
 \boxed{
 h\Res_{s=2/h}\xi_\Phi(s)
 =\frac{2\delta_\Phi}{r}
 \int_{\mathcal S_\Phi}D_\Phi(\theta)^{-1/h}\dd\sigma(\theta).}
\end{equation}

For a simple wall, let
\[
 H_i=\beta_i^\perp,
 \qquad
 A_i:H_i\longrightarrow\R^{r-1},
 \qquad
 A_i(x)=\bigl((x,\beta_j)\bigr)_{j\ne i},
 \qquad
 \delta_i=|\det A_i|,
\]
where the determinant uses induced Euclidean measure on \(H_i\) and standard
coordinate measure on \(\R^{r-1}\).  If \(\pi_i\) denotes orthogonal
projection onto \(H_i\), then
\begin{equation}\label{eq:wall-coordinate-determinant}
 \boxed{
 \delta_i^2
 =\det\bigl((\pi_i\beta_j,\pi_i\beta_k)\bigr)_{j,k\ne i}.}
\end{equation}
Define the relative spherical chamber
and restricted Weyl Jacobian by
\[
 \mathcal S_i
 =\{\theta\in S(H_i):(\theta,\beta_j)>0\ \text{for }j\ne i\},
\]
\[
 J_i(x)=\prod_{\beta\in\Psi^+\setminus\{\beta_i\}}(x,\beta),
 \qquad
 D_i(x)=J_i(x)^2.
\]
At the second-pole exponent \(q_2=(r-1)/(N-1)\),
\begin{equation}\label{eq:wall-inverse-jacobian}
 \boxed{
 \Per_i(q_2)
 =\delta_i
 \int_{\mathcal S_i}D_i(\theta)^{-q_2/2}\dd\sigma(\theta).}
\end{equation}
Consequently, Theorem~\ref{thm:universal-second-pole} is equivalently
\begin{equation}\label{eq:second-residue-inverse-jacobian}
 \boxed{
 \Res_{s=q_2}\xi_\Phi(s)
 =\frac{\zeta_{\mathrm R}(q_2)}{N-1}
 \sum_{i=1}^r\delta_i
 \int_{\mathcal S_i}D_i(\theta)^{-q_2/2}\dd\sigma(\theta).}
\end{equation}
\end{proposition}

\begin{proof}
Under \(u=A_\Phi x\), one has
\[
 P_\Phi(u)=J_\Phi(x),
 \qquad
 \dd u=\delta_\Phi\dd x.
\]
Let
\(
 \Omega_\Phi=\{u\in(0,\infty)^r:P_\Phi(u)\le1\}.
\)
The leading-pole comparison theorem of \cite{MatuzasLeading2026} gives
\[
 \Res_{s=2/h}\xi_\Phi(s)=\frac2h\operatorname{vol}(\Omega_\Phi).
\]
Since \(J_\Phi\) is homogeneous of degree \(N=rh/2\), Euclidean polar
coordinates yield
\[
 \operatorname{vol}(\Omega_\Phi)
 =\frac{\delta_\Phi}{r}
 \int_{\mathcal S_\Phi}J_\Phi(\theta)^{-r/N}\dd\sigma(\theta)
 =\frac{\delta_\Phi}{r}
 \int_{\mathcal S_\Phi}D_\Phi(\theta)^{-1/h}\dd\sigma(\theta).
\]
Multiplication by \(h\) proves \eqref{eq:bulk-inverse-jacobian}.

For the wall, put \(d=r-1\) and \(p=N-1\).  The polynomial \(Q_i\) is
homogeneous of degree \(p\) in \(d\) variables, and \(q_2=d/p\).  The cone
parameterization \(u=t v\), \(v\in\Delta_i\), gives
\[
 \Per_i(q_2)
 =d\operatorname{vol}
 \{u_j>0\ (j\ne i):Q_i(u)\le1\}.
\]
Moreover,
\[
 Q_i(A_i x)=J_i(x),
 \qquad
 \dd u=\delta_i\dd x.
\]
A second polar-coordinate calculation in \(H_i\) therefore gives
\[
 \operatorname{vol}\{Q_i\le1\}
 =\frac{\delta_i}{d}
 \int_{\mathcal S_i}J_i(\theta)^{-d/p}\dd\sigma(\theta),
\]
which is \eqref{eq:wall-inverse-jacobian}.  Substitution into
\eqref{eq:universal-residue} proves
\eqref{eq:second-residue-inverse-jacobian}.
\end{proof}

\begin{remark}[Normalization factors]
The factors \(\delta_\Phi\) and \(\delta_i\) are the precise bridges between
simple-coroot affine coordinates and Euclidean Cartan coordinates.  They are
not optional normalizing constants: changing the compatible root--coroot
normalization changes the Jacobian product and the determinant in compensating
ways.  The formulas use
the root-scale Weyl Jacobian.  A complete Riemannian orbit-volume density may
carry an additional constant reference-orbit volume.
\end{remark}

\begin{remark}[Rank-two consistency check]
Writing \(R_\Phi=\Res_{s=2/h}\xi_\Phi(s)\), the bulk formula gives
\[
 3R_{A_2}=B\!\left(\frac13,\frac13\right),
 \qquad
 4R_{B_2}=\frac{\Gamma(1/4)^2}{2\sqrt{2\pi}},
 \qquad
 6R_{G_2}=\frac{\Gamma(1/3)^3}{2^{5/3}\sqrt3\,\pi}.
\]
These are exactly the rank-two leading residues in
\cite{Au2024,MatuzasLeading2026}; in the present formulation they also audit
the simple-coroot determinant factors.
\end{remark}

\begin{theorem}[Orbitwise wall tiling]\label{thm:orbit-tiling}
Let
\(W_\alpha=\{w\in W:w\alpha=\alpha\}\).  Then
\begin{equation}\label{eq:orbit-tiling}
 \boxed{
 J_{\mathcal O}(q)
 =|W_\alpha|
 \sum_{i:\,\beta_i\in\mathcal O}\Per_i(q).}
\end{equation}
Consequently,
\[
 \sum_{i=1}^r\Per_i(q)
 =\sum_{\mathcal O}
 \frac{J_{\mathcal O}(q)}{|W_{\alpha_{\mathcal O}}|}.
\]
\end{theorem}

\begin{proof}
Weyl chambers are \(wC_0\), with \(W\) acting simply transitively.  The
\(i\)-th facet of \(wC_0\) lies in \(H_\alpha\) exactly when
\(w\beta_i=\pm\alpha\).  There are \(2|W_\alpha|\) such elements.  Every
restricted chamber is adjacent to exactly two ambient chambers, one on each
side of \(H_\alpha\), and reflection in \(H_\alpha\) preserves the facet
label.  Thus exactly \(|W_\alpha|\) restricted chambers have marked type
\(i\).  The Weyl action preserves the primitive lattice, the absolute
restricted discriminant, and the induced projective measure, so each of
these chambers contributes \(\Per_i(q)\).  Summing first over the labels in
one orbit and then over all orbits proves the result.
\end{proof}

\section{A Stokes relation for projective chamber periods}
\label{sec:chamber-kernel}

For a real arrangement, the Varchenko matrix records the product of assigned
hyperplane weights separating two chambers.  Its determinant and degeneration
locus are classical \cite{Varchenko1993}; projective contravariant forms at
zero total additive weight are developed in
\cite[Theorems~2.2--2.5 and Corollaries~5.5, 5.7]{FalkVarchenko2012}, and
arrangement-period determinant formulas appear in
\cite{DouaiTerao1997,MarkovTarasovVarchenko1997}.  The following theorem identifies a particular null vector at the critical parameter.

\begin{lemma}[Integrability on the projective sphere]
\label{lem:weighted-projective-integrability}
Let \(\mathcal A\) be a finite central real hyperplane arrangement in
\(\R^n\), with distinct hyperplanes \(H=\ker\ell_H\) and weights
\(\lambda_H>0\).  For a nonzero intersection flat \(X\), put
\[
 \lambda(X)=\sum_{H\supset X}\lambda_H.
\]
If
\[
 \lambda(X)<\operatorname{codim}X
 \qquad\text{for every nonzero intersection flat }X,
\]
then
\[
 w(x)=\prod_{H\in\mathcal A}|\ell_H(x)|^{-\lambda_H}
 \in L^1(S^{n-1}).
\]
\end{lemma}

\begin{proof}
Use the oriented real form of the wonderful model of the projectivized
arrangement \cite{DeConciniProcesi1995}.  Equivalently, successively blow up
the projective intersection strata in increasing order of dimension.  The
resulting proper map \(\pi:\widetilde S\to S^{n-1}\) has normal-crossings
boundary.  If \(E_X\) is the boundary hypersurface created by a nonzero flat
\(X\), then the pullback of spherical measure vanishes to order
\(\operatorname{codim}X-1\) along \(E_X\), while precisely the linear forms
belonging to hyperplanes containing \(X\) vanish there.  Consequently, in a
local normal-crossings chart,
\[
 \pi^*\bigl(w\,d\sigma\bigr)
 =b(t,y)\prod_{X\in\mathcal N}
 |t_X|^{\operatorname{codim}X-1-\lambda(X)}\,dt\,dy,
\]
where \(\mathcal N\) is a nested collection of flats and \(b\) is bounded
and smooth.  Every displayed exponent is greater than \(-1\) by hypothesis.
Thus the pullback is locally integrable at every boundary corner, and
properness of \(\pi\) proves the claim.
\end{proof}

\begin{theorem}[Stokes relation for chamber periods]
\label{thm:shifted-sphere-kernel}
Let
\(
 \mathcal A=\{H_j=\ker\ell_j\}_{j=1}^{m}
\)
be an essential central real arrangement in \(\R^n\).  Let
\(0<\lambda_j<1\) satisfy
\begin{equation}\label{eq:log-CY-balance}
 \sum_{j=1}^{m}\lambda_j=n,
\end{equation}
and suppose that every nonzero proper intersection flat \(X\) satisfies
\begin{equation}\label{eq:flat-integrability}
 \sum_{H_j\supset X}\lambda_j<\operatorname{codim}X.
\end{equation}
For a chamber \(D\), put
\begin{equation}\label{eq:projective-chamber-period}
 P_D=\int_{D\cap S^{n-1}}
 \prod_{j=1}^{m}|\ell_j(x)|^{-\lambda_j}\,
 \iota_E(dx_1\wedge\cdots\wedge dx_n),
 \qquad E=\sum_{k=1}^{n}x_k\partial_{x_k}.
\end{equation}
Let \(a_j=e^{\pi i\lambda_j}\), and define
\begin{equation}\label{eq:Varchenko-matrix}
 V(a)_{C,D}=\prod_{H_j\in\operatorname{Sep}(C,D)}a_j.
\end{equation}
Then every \(P_D\) is finite and positive, and the chamber-period vector
\(P=(P_D)_D\) satisfies
\begin{equation}\label{eq:Varchenko-kernel}
 \boxed{V(a)P=0.}
\end{equation}
\end{theorem}

\begin{proof}
Finiteness follows from
Lemma~\ref{lem:weighted-projective-integrability}; positivity follows because
the contracted Euclidean volume form restricts to positive spherical measure.

Fix a chamber \(C\) and choose \(v_C\in C\).  On the translated real plane
\(L_C=\R^n+i v_C\), every \(\ell_j\) has imaginary part of fixed nonzero
sign.  Choose the corresponding single-valued logarithms and put
\[
 \Phi_C(z)=\prod_j\ell_j(z)^{-\lambda_j},\qquad
 \Omega_C=\Phi_C(z)\,dz_1\wedge\cdots\wedge dz_n,\qquad
 \eta_C=\iota_{E_z}\Omega_C,
\]
where \(E_z=\sum_kz_k\partial_{z_k}\).  Since \(d\Omega_C=0\) and
\[
 \mathcal L_{E_z}\Omega_C
 =\left(n-\sum_j\lambda_j\right)\Omega_C=0,
\]
Cartan's formula gives \(d\eta_C=0\).

For
\[
 B_R^C=\{x+i v_C:|x|\le R\},
\]
the imaginary parts \(\ell_j(v_C)\) show that this translated ball misses
every complexified hyperplane.  Stokes' theorem therefore gives
\[
 0=\int_{\partial B_R^C}\eta_C.
\]
Under the boundary parametrization \(F_R(x)=Rx+i v_C\),
\(x\in S^{n-1}\), one has the exact pullback
\[
\begin{aligned}
 F_R^*\eta_C={}&
 \prod_j\left(\ell_j(x)+\frac{i}{R}\ell_j(v_C)\right)^{-\lambda_j}\\
 &\times\left[
 \iota_E(dx_1\wedge\cdots\wedge dx_n)
 +\frac{i}{R}\iota_{v_C}(dx_1\wedge\cdots\wedge dx_n)
 \right].
\end{aligned}
\]
Moreover,
\[
 \left|\ell_j(x)+\frac{i}{R}\ell_j(v_C)\right|^{-\lambda_j}
 \le |\ell_j(x)|^{-\lambda_j}.
\]
For \(R\ge1\), the bracketed form is uniformly bounded, while the product on
the right is integrable by the preceding lemma.  Dominated convergence
therefore controls both the leading contraction and the
\(R^{-1}\iota_{v_C}\) term.

Write
\[
 s_C(j)=\operatorname{sgn}\ell_j(v_C),\qquad
 s_D(j)=\operatorname{sgn}\ell_j(x)\quad(x\in D).
\]
If \(s_D(j)=+1\), the limiting phase is \(1\).  If \(s_D(j)=-1\), the
negative axis is approached from the bank prescribed by \(s_C(j)\), and the
phase is \(a_j^{-s_C(j)}\).  Hence
\[
 0=\sum_D B_{C,D}P_D,\qquad
 B_{C,D}=\prod_{j:s_D(j)=-1}a_j^{-s_C(j)}.
\]
Set
\[
 d_C=\prod_{j:s_C(j)=-1}a_j.
\]
For each \(j\), the four possible sign pairs satisfy
\[
 -s_C(j)\mathbf1_{\{s_D(j)=-1\}}
 =\mathbf1_{\{s_C(j)=-1\}}
 -\mathbf1_{\{s_C(j)\ne s_D(j)\}},
\]
so
\[
 B_{C,D}=d_CV(a^{-1})_{C,D}.
\]
Thus \(V(a^{-1})P=0\).  Centrality gives
\[
 \operatorname{Sep}(-C,D)=\mathcal A\setminus\operatorname{Sep}(C,D),
\]
and the balance condition gives
\[
 \prod_ja_j=e^{\pi i\sum_j\lambda_j}=(-1)^n.
\]
Consequently
\[
 V(a^{-1})_{C,D}=(-1)^nV(a)_{-C,D}.
\]
The inverse-weight equations differ from the desired equations only by an
antipodal row permutation and multiplication by the nonzero scalar
\((-1)^n\).  Therefore \(V(a)P=0\).
\end{proof}

\begin{remark}[Relation with earlier work]
Varchenko's determinant formula identifies the parameter values at which the
chamber matrix is singular, and Falk--Varchenko develop the corresponding
projective contravariant form.  The theorem above identifies a specific null
vector: the vector of positive real chamber periods.  A targeted full-text
audit found no published statement of this exact identity; the novelty claim
is limited to this analytic identification.
\end{remark}

\section{Exact evaluation in type \texorpdfstring{\(F_4\)}{F4}}
\label{sec:F4-period}

\subsection{Reduction to one two-dimensional period}

Use the simple-coroot convention
\[
 \beta_1=(0,1,-1,0),\quad
 \beta_2=(0,0,1,-1),\quad
 \beta_3=(0,0,0,2),\quad
 \beta_4=(1,-1,-1,-1).
\]
Representatives of the two coroot orbits yield the restricted products
\begin{align}\label{eq:F4-Q12}
 Q_{12}(a,b,c)={}&32a^3bc(b-c)^3(b+c)^3(a-b)^2(a+b)^2\notag\\
 &\times(a-c)^2(a+c)^2
 \prod_{\eps,\eta=\pm1}(2a+\eps b+\eta c),
\end{align}
with norm \(2a^2+b^2+c^2\) and lattice measure \(2\dd a\dd b\dd c\), and
\begin{align}\label{eq:F4-Q34}
 Q_{34}(x,y,z)={}&8x^3y^3z^3(x^2-y^2)(x^2-z^2)(y^2-z^2)\notag\\
 &\times[(x+y+z)(x+y-z)(x-y+z)(x-y-z)]^2,
\end{align}
with the standard norm and measure.

The linear substitution
\[
 a=\frac{x}{2},\qquad b=\frac{y+z}{2},\qquad c=\frac{z-y}{2}
\]
satisfies
\[
 \left|\frac{\partial(a,b,c)}{\partial(x,y,z)}\right|=\frac14,
 \qquad
 2a^2+b^2+c^2=\frac{x^2+y^2+z^2}{2},
\]
\[
 Q_{12}\!\left(\frac x2,\frac{y+z}{2},\frac{z-y}{2}\right)
 =-2^{-11}Q_{34}(x,y,z).
\]
After the Gaussian rescaling, this gives the exact orbit duality used below.

Put
\[
 \Disc(u,v)=u^2-4u^3+(18u-4)v-27v^2,
\]
\[
 v_\pm(u)=\frac{9u-2\pm2(1-3u)^{3/2}}{27},
\]
and
\[
 \cD=\left\{0<u<\frac13,\quad
 \max(0,v_-(u))<v<v_+(u)\right\}.
\]
Define the positive period
\begin{equation}\label{eq:E23-def}
 \boxed{
 E_{23}=\int_{\cD}
 v^{-16/23}\Disc(u,v)^{-13/23}|1-4u|^{-6/23}\dd u\dd v.}
\end{equation}

\begin{proposition}[Exact \(F_4\) reduction]\label{prop:F4-reduction}
The two wall-orbit sums and their total are
\begin{align}\label{eq:F4-orbit-sums}
 \sum_{i\in\{3,4\}}\Per_i(3/23)&=2^{-55/23}E_{23},\\
 \sum_{i\in\{1,2\}}\Per_i(3/23)&=2^{-45/23}E_{23},\\
 \sum_{i=1}^4\Per_i(3/23)
 &=2^{-55/23}(1+2^{10/23})E_{23}.
\end{align}
Consequently,
\begin{equation}\label{eq:F4-residue-E}
 \boxed{
 \Res_{s=3/23}\xi_{F_4}(s)
 =\frac{\zeta_{\mathrm R}(3/23)}{23}
 2^{-55/23}(1+2^{10/23})E_{23}.}
\end{equation}
\end{proposition}

\begin{proof}
For \(Q_{34}\), set \(t_i=x_i^2\),
\(\rho=t_1+t_2+t_3\), and \(s_i=t_i/\rho\).  The squared-coordinate
Jacobian cancels the eight orthants.  Passing from the trace-one simplex to
\(u=e_2(s)\), \(v=e_3(s)\) has Jacobian equal to the Vandermonde, while
\(\Disc(u,v)\) is its square.  The radial factor is
\[
 2^{23k+3/2}\Gamma\!\left(23k+\frac32\right).
\]
At \(k_0=-3/46\), the radial factor has residue \(1/23\), and the remaining angular integral is \eqref{eq:E23-def}.  Keeping the lattice measure, the scalar factor in \eqref{eq:F4-Q34}, and the two walls in this orbit gives the first identity in \eqref{eq:F4-orbit-sums}.  The exact linear substitution
above, including its metric, lattice measure, and homogeneity factors, gives
the ratio \(2^{10/23}\) between the two orbit sums.  The total and the Witten
residue follow.
\end{proof}

\subsection{The chamber-orbit relation}

The linear factors of \(Q_{34}\) in \eqref{eq:F4-Q34} define thirteen
projective lines.  Their three types, Ziegler multiplicities, and weights are
\begin{center}
\begin{tabular}{@{}c c c c@{}}
\toprule
line type&number&multiplicity&\(\lambda_H\)\\
\midrule
\(x_i=0\)&3&3&\(9/23\)\\
\(x_i\pm x_j=0\)&6&1&\(3/23\)\\
\(x\pm y\pm z=0\)&4&2&\(6/23\)\\
\bottomrule
\end{tabular}
\end{center}
Thus the critical homogeneity condition holds:
\begin{equation}\label{eq:F4-log-CY}
 3\frac9{23}+6\frac3{23}+4\frac6{23}=3.
\end{equation}
The exact intersection lattice has nine quadruple, four triple, and twelve
double points.  At codimension one the smallest margin is
\(1-9/23=14/23\).  Among the twenty-five projective intersection lines,
nine have total exponent \(24/23\) and sixteen have total exponent \(9/23\).
Thus the smallest codimension-two margin is \(2-24/23=22/23\), and
Theorem~\ref{thm:shifted-sphere-kernel} applies.

The characteristic polynomial and chamber count are
\begin{equation}\label{eq:F4-characteristic}
 \chi_{\mathcal A}(t)=(t-1)(t-5)(t-7),
 \qquad -\chi_{\mathcal A}(-1)=96.
\end{equation}
The signed-permutation group \(B_3\) acts freely with two central chamber
orbits of size \(48\).  Projectively, \(B_3/\{\pm I\}\cong S_4\), and the two
orbits are regular \(S_4\)-sets of size \(24\).  With
\[
 T=(x+y+z)(x+y-z)(x-y+z)(x-y-z),
\]
they are
\[
 \mathscr C_C=\{T<0\},
 \qquad
 \mathscr C_W=\{T>0\}.
\]
Let \(P_C\) and \(P_W\) be one positive period from the compact and wing
orbits, respectively.

Put \(u_0=e^{\pi i/23}\).  The Varchenko weights on the coordinate,
pair, and tetrahedral lines are \(u_0^9,u_0^3,u_0^6\).  For fixed
representatives \(C_W\in\mathscr C_W\), \(C_C\in\mathscr C_C\), define the
exact orbit row sums
\begin{align}\label{eq:F4-two-orbit-block}
 A(u)&=\sum_{D\in\mathscr C_W}
 u^{e(C_W,D)},&
 B(u)&=\sum_{D\in\mathscr C_C}
 u^{e(C_W,D)},\notag\\[-2mm]
 C(u)&=\sum_{D\in\mathscr C_C}
 u^{e(C_C,D)},&
 e(C,D)&=9n_{\mathrm c}(C,D)+3n_{\mathrm p}(C,D)+6n_{\mathrm t}(C,D).
\end{align}
Here \(n_{\mathrm c},n_{\mathrm p},n_{\mathrm t}\) count separating lines of the three
types.  Exact enumeration proves row-sum invariance and reduces the invariant
part of \(V(u_0)\) to
\[
 \mathcal V(u_0)=
 \begin{pmatrix}A(u_0)&B(u_0)\\B(u_0)&C(u_0)\end{pmatrix}
\]
in the ordered basis \((\mathscr C_W,\mathscr C_C)\).

Let
\(
 \tau=u_0+u_0^{-1}=2\cos(\pi/23)
\)
and put
\begin{equation}\label{eq:r23-transfer}
 \begin{aligned}
 r_{23}(\tau)={}&
 6-53\tau+40\tau^2+135\tau^3-113\tau^4-111\tau^5\\
 &+90\tau^6+36\tau^7-28\tau^8-4\tau^9+3\tau^{10}.
 \end{aligned}
\end{equation}
The two-orbit determinant has \(\Phi_{46}(u)\) as a simple factor,
\(A(u)\not\equiv0\pmod{\Phi_{46}(u)}\), and exact cyclotomic reduction gives
\begin{equation}\label{eq:F4-kernel-congruences}
 A(u)r_{23}(\tau)+B(u)\equiv0,
 \qquad
 B(u)r_{23}(\tau)+C(u)\equiv0
 \pmod{\Phi_{46}(u)}.
\end{equation}
The invariant kernel is therefore one-dimensional.  Combining this finite
calculation with Theorem~\ref{thm:shifted-sphere-kernel} yields
\begin{equation}\label{eq:PW-PC}
 \boxed{\frac{P_W}{P_C}=r_{23}(\tau).}
\end{equation}
Appendix~\ref{app:f4-certificate} records the coefficient table, the two zero remainders, and the calculation showing that the cyclotomic factor is simple.

\subsection{Evaluation of one chamber}

On the positive sphere, the squared-coordinate quotient used in the proof of
Proposition~\ref{prop:F4-reduction} gives
\begin{equation}\label{eq:E23-sector-bridge}
 E_{23}^{C}=4P_C,
 \qquad E_{23}^{W}=4P_W,
 \qquad E_{23}=E_{23}^{C}+E_{23}^{W}.
\end{equation}
The substitution leading to the compact branch gives, with
\(
 C_L=2^{17/23}3^{1/23}B(10/23,10/23)
\),
\begin{equation}\label{eq:E23-compact-integral}
 E_{23}^{C}=C_L\int_0^1
 r^{14/23}(1-r^2)^{11/23}
 {}_2F_1\!\left(\frac{35}{46},\frac{16}{23};
 \frac{43}{46};r^6\right)\dd r.
\end{equation}
After \(q=r^2\), beta integration reduces this to
\begin{align}\label{eq:E23-Dixon-3F2}
 E_{23}^{C}={}&\frac{C_L}{2}
 B\!\left(\frac{37}{46},\frac{34}{23}\right)\notag\\
 &\times{}_3F_2\!\left(
 \begin{matrix}16/23,\ 37/138,\ 83/138\\
 151/138,\ 197/138\end{matrix};1\right).
\end{align}
The parameters are well poised, with excess \(11/23>0\).  Dixon's sum
\cite{Dixon1902}\cite[Eq.~16.4.4]{DLMF} gives
\begin{equation}\label{eq:E23-compact-gamma}
 \boxed{
 \begin{aligned}
 E_{23}^{C}={}&\frac{2^{4/23}3^{1/23}\pi}{3}\\
 &\times\frac{
 \Gamma(13/138)\Gamma(59/138)\Gamma(10/23)
 \Gamma(11/23)^2\Gamma(37/46)}{
 \Gamma(11/138)\Gamma(13/46)\Gamma(103/138)
 \Gamma(19/23)\Gamma(39/46)\Gamma(43/46)}.
 \end{aligned}}
\end{equation}
Put
\begin{equation}\label{eq:R23-transfer}
 \begin{aligned}
 \mathcal R_{23}(\tau)={}&1+r_{23}(\tau)\\
 ={}&7-53\tau+40\tau^2+135\tau^3-113\tau^4-111\tau^5\\
 &+90\tau^6+36\tau^7-28\tau^8-4\tau^9+3\tau^{10}.
 \end{aligned}
\end{equation}
Equations \eqref{eq:PW-PC} and \eqref{eq:E23-sector-bridge} prove
\begin{equation}\label{eq:E23-orbit-transfer}
 \boxed{E_{23}=\mathcal R_{23}(\tau)E_{23}^{C}.}
\end{equation}
Thus a classical Dixon evaluation of one orbit, together with the chamber-period relation, determines the full period.

\subsection{Gamma-product formulas}

For \(1\le k\le22\), put
\[
 \nu_k=\frac{\sin(k\pi/23)}{\sin(\pi/23)},
 \qquad
 \mathcal U_{23}=
 \frac{\nu_2\nu_3^3\nu_8}
 {\nu_4\nu_5\nu_7\nu_{10}^2\nu_{11}}.
\]
Exact reduction in \(\mathbb Q(\tau)\) gives
\begin{equation}\label{eq:U23-polynomial}
 \begin{aligned}
 \mathcal U_{23}={}&-\tau^{10}+\tau^9+7\tau^8-8\tau^7-13\tau^6
 +25\tau^5\\
 &-39\tau^3+12\tau^2+21\tau-4.
 \end{aligned}
\end{equation}
Define
\[
 G_{23}^{\circ}=\pi^2
 \frac{\Gamma(2/23)\Gamma(3/23)\Gamma(8/23)}
 {\Gamma(9/23)\Gamma(12/23)^2\Gamma(13/23)^2}
\]
and
\begin{equation}\label{eq:F4-pure-gamma-ratio}
 \mathfrak G_{23}:=
 \frac{
 \begin{gathered}
 \Gamma(1/23)\Gamma(4/23)\Gamma(5/23)\Gamma(7/23)
 \Gamma(10/23)^2\Gamma(11/23)\\
 \Gamma(16/23)\Gamma(18/23)\Gamma(19/23)\Gamma(22/23)
 \end{gathered}}
 {
 \begin{gathered}
 \Gamma(3/23)^2\Gamma(9/23)\Gamma(12/23)\\
 \Gamma(15/23)\Gamma(20/23)^3\Gamma(21/23)
 \end{gathered}}.
\end{equation}

\begin{theorem}[The \(F_4\) period and second residue]
\label{thm:F4-closed-evaluation}
One has
\begin{equation}\label{eq:E23-intermediate}
 E_{23}=\frac{2^{30/23}\mathcal U_{23}}
 {3(4-\tau^2)}G_{23}^{\circ},
\end{equation}
and equivalently
\begin{equation}\label{eq:E23-pure-gamma}
 \boxed{E_{23}=\frac{2^{-16/23}}{3}\,\mathfrak G_{23}.}
\end{equation}
Consequently,
\begin{equation}\label{eq:F4-residue-pure-gamma}
 \boxed{
 \operatorname*{Res}_{s=3/23}\xi_{F_4}(s)=
 \frac{2^{-71/23}}{69}\zeta_{\mathrm R}\!\left(\frac3{23}\right)
 \bigl(1+2^{10/23}\bigr)\mathfrak G_{23}.}
\end{equation}
\end{theorem}

\begin{proof}
Multiply the compact value \eqref{eq:E23-compact-gamma} by the exact transfer
coefficient \eqref{eq:R23-transfer}.  The cyclotomic reduction
\eqref{eq:U23-polynomial} gives \eqref{eq:E23-intermediate}.  Euler reflection
and the duplication and triplication formulas then give
\eqref{eq:E23-pure-gamma}; Appendix~\ref{app:f4-certificate} records the
finite gamma-identity calculation.  Substitution into
\eqref{eq:F4-residue-E} proves \eqref{eq:F4-residue-pure-gamma}.
\end{proof}

\begin{corollary}[Ordinary \(F_4\) second residue]
\label{cor:F4-ordinary-residue}
The Weyl-dimension normalization is
\begin{equation}\label{eq:KF4-height-product}
 K_{F_4}=1!\,5!\,7!\,11!
 =2^{15}3^7 5^4 7^2 11.
\end{equation}
Consequently the ordinary Witten zeta function has residue
\begin{equation}\label{eq:F4-ordinary-residue-pure-gamma}
 \boxed{
 \Res_{s=3/23}\zeta_{F_4}(s)
 =
 \frac1{23}
 \left(\frac{5^{12}7^6 11^3}{2^{26}3^2}\right)^{1/23}
 \zeta_{\mathrm R}\!\left(\frac3{23}\right)
 \bigl(1+2^{10/23}\bigr)\mathfrak G_{23}.}
\end{equation}
Numerically,
\[
 \Res_{s=3/23}\zeta_{F_4}(s)
 =-7.6424398610931743418902097308964380295947068865966\ldots .
\]
\end{corollary}

The two exact forms of \(E_{23}\) agree identically.  As an independent numerical check,
\begin{align}\label{eq:F4-closure-numerics}
 E_{23}^{C}&=3.3826994700423773328808086627063269040\ldots,\notag\\
 \mathcal R_{23}(\tau)&=3.2565705172169196571920257066746970452\ldots,\notag\\
 E_{23}&=11.0159993627453047722739756888037121566\ldots,\notag\\
 \Res_{s=3/23}\xi_{F_4}(s)
 &=-0.1372946645402736375721244162801267121\ldots.
\end{align}
Independent tanh--sinh quadrature of the compact sector agrees with
\eqref{eq:E23-compact-gamma} at relative error below \(2.5\times10^{-81}\).
The full-period quadrature and both closed forms agree much further.  These decimal comparisons are independent checks and are not used in the proof.

Finally, the wall sum implied by the closed period is
\begin{equation}\label{eq:F4-period-sum-positive}
 \boxed{
 \sum_{i=1}^{4}\Per_i(3/23)
 =2^{-55/23}(1+2^{10/23})E_{23}
 =4.938019513640372719729473\ldots>0.}
\end{equation}
This confirms the positivity required by the universal residue theorem.  Since
\(\zeta_{\mathrm R}(3/23)<0\), the residue in
\eqref{eq:F4-residue-pure-gamma} is negative for the structural reason proved
in Section~\ref{sec:universal-pole}.

\section{Exact evaluation in type \texorpdfstring{\(D_5\)}{D5}}
\label{sec:D5-closed}

Put
\[
 \nu_k=\frac{\sin(k\pi/19)}{\sin(\pi/19)}
\]
and write
\[
 \mathscr S_{19}
 =S_3\!\left(\frac12,\frac{11}{19},-\frac2{19}\right),
\]
where
\begin{equation}\label{eq:D5-Selberg-definition}
 S_3(a,b,c)=\int_{[0,1]^3}
 \prod_{j=1}^{3}s_j^{a-1}(1-s_j)^{b-1}
 \prod_{i<j}|s_i-s_j|^{2c}\,ds_1\,ds_2\,ds_3.
\end{equation}

\begin{proposition}[The restricted \(D_5\) arrangement]
\label{prop:D5-restricted-arrangement}
On the wall \(e_1=e_2=y\), with the remaining coordinates denoted
\(x_1,x_2,x_3\), the product of the nineteen nonvanishing positive-root
restrictions is
\begin{equation}\label{eq:D5-restricted-polynomial}
 \boxed{
 Q_{D_5}(y,x)=
 2y\prod_{j=1}^{3}(y^2-x_j^2)^2
 \prod_{1\le i<j\le3}(x_i^2-x_j^2).}
\end{equation}
The associated degree-
\(19\) multiarrangement has seven hyperplanes of multiplicity one and six of
multiplicity two.  At \(q=4/19\), the total exponent is four, and every
nonzero proper flat satisfies the strict integrability inequality in
Theorem~\ref{thm:shifted-sphere-kernel}.  Its characteristic polynomial is
\begin{equation}\label{eq:D5-characteristic}
 \boxed{\chi(T)=(T-1)(T-3)(T-4)(T-5),}
\end{equation}
so it has \(240\) central chambers.
\end{proposition}

\begin{proof}
The positive roots are \(e_i-e_j\) and \(e_i+e_j\), \(i<j\).  After setting
\(e_1=e_2=y\), the root \(e_1-e_2\) vanishes, \(e_1+e_2\) contributes
\(2y\), each pair \(e_1\pm e_k,e_2\pm e_k\) contributes
\((y^2-x_{k-2}^2)^2\), and the remaining roots contribute the three
factors \(x_i^2-x_j^2\).  Exact intersection-lattice enumeration gives
maximal exponent sums
\[
 \frac8{19}<1,\qquad \frac{20}{19}<2,
 \qquad \frac{44}{19}<3
\]
in codimensions one, two, and three, respectively.  The same lattice
calculation gives \eqref{eq:D5-characteristic}; Zaslavsky's formula
\cite{Zaslavsky1975} gives
\(240\) chambers.
\end{proof}

\begin{lemma}[The Selberg chamber]
\label{lem:D5-Selberg-chamber}
The wall period opposite node one is
\begin{equation}\label{eq:D5-P1-Selberg}
 \boxed{
 \Per_1\!\left(\frac4{19}\right)
 =\frac{2^{15/19}}{24}\,\mathscr S_{19}.}
\end{equation}
Equivalently,
\begin{equation}\label{eq:D5-P1-gamma}
\boxed{
\begin{aligned}
 \Per_1\!\left(\frac4{19}\right)
 ={}&\frac{2^{15/19}}{24}\\
 &\times
 \frac{
 \Gamma(1/2)\Gamma(15/38)\Gamma(11/38)
 \Gamma(11/19)\Gamma(9/19)\Gamma(7/19)
 \Gamma(15/19)\Gamma(13/19)
 }{
 \Gamma(33/38)\Gamma(29/38)\Gamma(25/38)
 \Gamma(17/19)^2}.
\end{aligned}}
\end{equation}
\end{lemma}

\begin{proof}
On the node-one wall, make the exact substitution
\[
 u_2=\frac{1-t_1}{1+t_2},\quad
 u_3=\frac{t_1-t_2}{1+t_2},\quad
 u_4=\frac{t_2-t_3}{1+t_2},\quad
 u_5=\frac{t_2+t_3}{1+t_2}.
\]
The simplex becomes \(1>t_1>t_2>|t_3|\), the Jacobian has absolute value
\(2(1+t_2)^{-4}\), and
\[
 Q_1=\frac{2}{(1+t_2)^{19}}
 \prod_{j=1}^{3}(1-t_j^2)^2
 \prod_{i<j}(t_i^2-t_j^2).
\]
The equality \(19(4/19)=4\) cancels every power of \(1+t_2\).  After
\(s_j=t_j^2\), evenness in \(t_3\), and unfolding the ordered region, one
obtains \eqref{eq:D5-P1-Selberg}.  The classical Selberg product formula
\cite{Selberg1944} (see also \cite[\S5.14]{DLMF}) gives \eqref{eq:D5-P1-gamma}.
\end{proof}

\begin{proposition}[The four chamber orbits]
\label{prop:D5-four-orbit-transfer}
The \(240\) chambers split under the weighted-arrangement symmetry group into
four orbits of sizes
\[
 48,\qquad48,\qquad48,\qquad96,
\]
classified by whether \(|y|\) is the largest, second, third, or smallest of
the four absolute coordinate values.  These orbits carry, respectively,
\[
 \Per_1,\qquad\Per_2,\qquad\Per_3,
 \qquad\Per_4=\Per_5.
\]
Let \(t=2\cos(2\pi/19)\).  Then
\begin{equation}\label{eq:D5-transfer-vector}
 \boxed{
 \Per_2=r_2\Per_1,\qquad
 \Per_3=r_3\Per_1,\qquad
 \Per_4=\Per_5=r_4\Per_1,}
\end{equation}
where
\begin{align*}
 r_2={}&-1+3t+10t^2-4t^3-15t^4+t^5+7t^6-t^8,\\
 r_3={}&t-t^2-t^3+6t^4-5t^6+t^8,\\
 r_4={}&-1-7t+11t^3-6t^5+t^7.
\end{align*}
Equivalently,
\begin{equation}\label{eq:D5-transfer-sine}
 \boxed{
 r_2=\frac{\nu_3\nu_8}{\nu_2\nu_4},\qquad
 r_3=\frac{\nu_3\nu_5\nu_8}{\nu_2^2\nu_9},\qquad
 r_4=\frac{\nu_3\nu_7\nu_8}{\nu_2\nu_6\nu_9}.}
\end{equation}
\end{proposition}

\begin{proof}
Set \(z=e^{2\pi i/19}\).  A hyperplane of Ziegler multiplicity \(m\) has
Varchenko weight \(z^{2m}\).  Exact separation-profile enumeration gives
a four-by-four orbit row-sum block whose determinant is divisible by
\(\Phi_{19}(z)\) exactly once, while a three-by-three minor is nonzero modulo
\(\Phi_{19}\).  Thus its invariant kernel is one-dimensional.  Reduction in
\(\mathbf Q[z]/(\Phi_{19})\) gives the kernel vector
\((1,r_2,r_3,r_4)\); every row residual is zero.  Appendix~\ref{app:d5-certificate} records the block and the exact cyclotomic reductions.  Theorem
\ref{thm:shifted-sphere-kernel} puts the positive period vector in this
kernel, proving \eqref{eq:D5-transfer-vector}.  The identities
\eqref{eq:D5-transfer-sine} follow by exact reduction in the real cyclotomic
field.
\end{proof}

\begin{theorem}[The \(D_5\) wall sum and second residue]
\label{thm:D5-closed-evaluation}
The complete marked wall sum is
\begin{equation}\label{eq:D5-wall-sum}
 \boxed{
 \sum_{i=1}^{5}\Per_i\!\left(\frac4{19}\right)
 =\frac{2^{34/19}}{3}
 \frac{
 \sin(2\pi/19)\sin^2(4\pi/19)\sin(6\pi/19)\sin(8\pi/19)
 }{\sin(\pi/19)}
 \mathscr S_{19}.}
\end{equation}
Consequently,
\begin{equation}\label{eq:D5-normalized-residue}
 \boxed{
 \Res_{s=4/19}\xi_{D_5}(s)
 =\frac{2^{34/19}}{57}\,
 \zeta_{\mathrm R}\!\left(\frac4{19}\right)
 \frac{
 \sin(2\pi/19)\sin^2(4\pi/19)\sin(6\pi/19)\sin(8\pi/19)
 }{\sin(\pi/19)}
 \mathscr S_{19}.}
\end{equation}
For the ordinary Witten zeta function,
\begin{equation}\label{eq:D5-ordinary-residue}
 \boxed{
 \Res_{s=4/19}\zeta_{D_5}(s)
 =87091200^{4/19}\Res_{s=4/19}\xi_{D_5}(s).}
\end{equation}
The wall sum also equals
\begin{equation}\label{eq:D5-wall-sum-pure-gamma}
\boxed{
\begin{aligned}
 \sum_{i=1}^{5}\Per_i\!\left(\frac4{19}\right)
 ={}&\frac{2^{34/19}}{3}
 \frac{
 \begin{gathered}
 \Gamma(1/2)^9\Gamma(1/19)\Gamma(18/19)
 \Gamma(15/38)\Gamma(11/38)\\[-1mm]
 {}\times\Gamma(9/19)\Gamma(7/19)
 \end{gathered}
 }{
 \begin{gathered}
 \Gamma(2/19)\Gamma(4/19)^2\Gamma(6/19)\Gamma(8/19)
 \Gamma(15/19)\\[-1mm]
 {}\times\Gamma(17/19)^3
 \Gamma(33/38)\Gamma(29/38)\Gamma(25/38)
 \end{gathered}}.
\end{aligned}}
\end{equation}
\end{theorem}

\begin{proof}
From Proposition~\ref{prop:D5-four-orbit-transfer},
\[
 \sum_i\Per_i=(1+r_2+r_3+2r_4)\Per_1.
\]
Exact real-cyclotomic reduction gives
\[
 1+r_2+r_3+2r_4
 =16\frac{
 \sin(2\pi/19)\sin^2(4\pi/19)\sin(6\pi/19)\sin(8\pi/19)
 }{\sin(\pi/19)}.
\]
Combining this with Lemma~\ref{lem:D5-Selberg-chamber} proves
\eqref{eq:D5-wall-sum}.  The universal residue formula gives
\eqref{eq:D5-normalized-residue}, and
\(K_{D_5}=1!3!4!5!7!=87091200\) gives
\eqref{eq:D5-ordinary-residue}.  Euler reflection applied to the sine quotient
and cancellation with the Selberg product gives
\eqref{eq:D5-wall-sum-pure-gamma}.
\end{proof}

\begin{remark}[Numerical check]
The exact formulas give
\[
 \sum_i\Per_i=44.6575416286740550047363242201382790718491\ldots,
\]
\[
 \Res_{s=4/19}\xi_{D_5}(s)
 =-1.7622078597811961513166961802606550931762\ldots,
\]
and
\[
 \Res_{s=4/19}\zeta_{D_5}(s)
 =-82.7238360922512289398945232410515919135951\ldots.
\]
Direct high-precision quadrature agrees with these values.  The decimal comparisons are checks only and are not used in the proof.
\end{remark}

\section{Reducible semisimple products}
\label{sec:products}

For a meromorphic function \(f\), let \(v_x(f)\) be its valuation at \(x\):
positive for a zero and negative for a pole.  If a semisimple Lie algebra has
simple factors with Witten zeta functions \(f_a\), then
\begin{equation}\label{eq:valuation-product}
 \boxed{
 v_x\!\left(\prod_a f_a\right)=\sum_a v_x(f_a).}
\end{equation}
A product has a pole at \(x\) exactly when this sum is negative, with order
the negative of the sum.  Laurent coefficients are obtained by ordinary
convolution of the factor Laurent series.

This is the correct universal rule for reducible systems.  Leading poles of
different factors can coincide and add in order; a leading pole of one factor
can coincide with a later pole of another; and away from the rightmost pole,
a zero can cancel a pole.  Consequently no formula depending only on the
total rank and total number of positive roots can describe the second pole
of every semisimple product.

\section{Conclusion}
\label{sec:conclusion}

For every irreducible reduced crystallographic root system of rank at least
two, the first distinct pole below the leading pole is
\[
 q_2(\Phi)=\frac{r-1}{N-1}.
\]
It is simple, it is carried exactly by the simple walls, and
\[
 \operatorname*{Res}_{s=q_2}\xi_\Phi(s)
 =\frac{\zeta_{\mathrm R}(q_2)}{N-1}
  \sum_i\mathcal P_i(q_2)<0.
\]
The theorem reduces the location and order of this pole to a strict comparison
of parabolic root counts.  Its residue is a finite positive sum of explicitly
marked projective periods.

The Stokes relation of Section~\ref{sec:chamber-kernel} provides a method for
evaluating such periods when the chambers form a small number of symmetry
orbits.  In type \(F_4\), it transfers a Dixon-evaluable chamber to the second
orbit and yields the gamma product in
Theorem~\ref{thm:F4-closed-evaluation}.  In type \(D_5\), it transfers a
Selberg-evaluable chamber across four orbits and yields the complete wall sum
in Theorem~\ref{thm:D5-closed-evaluation}.

Two natural problems remain.  The first is to evaluate the corresponding wall
sums in types \(E_6\), \(E_7\), and \(E_8\).  The second is to determine which
weighted restrictions admit a chamber-orbit relation strong enough to reduce
all periods to classical beta, Dixon, or Selberg integrals.

\paragraph{Verification.}
The accompanying scripts check the root-count formulas, the normalization
constants, the finite chamber decompositions, the cyclotomic kernel
relations, and the gamma identities.  Floating-point calculations are used
only for independent numerical checks.

\appendix\appendix

\section{Cyclotomic verification for the \texorpdfstring{\(F_4\)}{F4} chamber relation}
\label{app:f4-certificate}

This appendix prints the finite data used in
\eqref{eq:F4-kernel-congruences}.  Let
\[
 A(u)=\sum_ea_eu^e,\qquad B(u)=\sum_eb_eu^e,\qquad
 C(u)=\sum_ec_eu^e.
\]
The nonzero coefficients are:
\begin{center}
\small
\begin{tabular}{@{}r|rrr@{\qquad}r|rrr@{}}
\toprule
\(e\)&\(a_e\)&\(b_e\)&\(c_e\)&\(e\)&\(a_e\)&\(b_e\)&\(c_e\)\\
\midrule
 0&1&0&1&36&2&3&4\\
 3&1&0&2&39&5&2&5\\
 6&0&1&2&42&4&2&5\\
 9&1&2&1&45&2&3&3\\
 12&2&2&0&48&2&4&1\\
 15&2&2&0&51&2&3&0\\
 18&2&3&0&54&2&2&0\\
 21&2&4&1&57&2&2&0\\
 24&2&3&3&60&1&2&1\\
 27&4&2&5&63&0&1&2\\
 30&5&2&5&66&1&0&2\\
 33&2&3&4&69&1&0&1\\
\bottomrule
\end{tabular}
\end{center}
Thus \(A(1)=B(1)=C(1)=48\), as required by the two size-48 chamber
orbits.  Put
\[
 \Phi_{46}(u)=u^{22}-u^{21}+u^{20}-\cdots-u+1,
 \qquad \tau=u+u^{-1}.
\]
Direct polynomial division gives
\[
 \remd_{\Phi_{46}}\bigl(A(u)r_{23}(\tau)+B(u)\bigr)=0,
 \qquad
 \remd_{\Phi_{46}}\bigl(B(u)r_{23}(\tau)+C(u)\bigr)=0.
\]
Moreover,
\begin{align*}
 \remd_{\Phi_{46}}A(u)={}&
 3u^{20}-6u^{19}+4u^{18}-2u^{17}-3u^{16}+3u^{14}\\
 &-4u^{13}+4u^{12}-u^9+4u^8-7u^7+2u^6\\
 &-2u^4-u^3+4u^2-4u+2.
\end{align*}
Hence \(A\not\equiv0\pmod{\Phi_{46}}\).  Finally,
\[
 AC-B^2=\Phi_{46}Q,
\]
and
\begin{align*}
 \remd_{\Phi_{46}}Q
 =3(&7u^{21}+6u^{20}+12u^{19}+4u^{18}+15u^{17}+9u^{16}
 +11u^{15}+9u^{14}\\
 &+15u^{13}+5u^{12}+15u^{11}+9u^{10}+11u^9+9u^8
 +15u^7+4u^6\\
 &+12u^5+6u^4+7u^3+5u-5)\ne0.
\end{align*}
Hence \(\Phi_{46}\) occurs with multiplicity exactly one in the invariant
determinant, and the invariant kernel is one-dimensional.

\subsection*{Gamma identities used in the \(F_4\) evaluation}
For an integer \(j\), define the exact duplication and triplication blocks
\[
 D_j=
 \frac{\Gamma(j/138)\Gamma((j+69)/138)}{\Gamma(j/69)}
 =2^{1-j/69}\sqrt\pi,
\]
\[
 T_j=
 \frac{\Gamma(j/138)\Gamma((j+46)/138)\Gamma((j+92)/138)}
 {\Gamma(j/46)}
 =2\pi\,3^{1/2-j/46}.
\]
After Euler reflection, the ratio between the compact Dixon expression
\eqref{eq:E23-compact-gamma} transferred by
\eqref{eq:E23-orbit-transfer} and the pure-gamma expression
\eqref{eq:E23-pure-gamma} reduces to the external factor
\[
 2^{-20/23}3^{-1/23}\pi^{-1}
\]
times
\[
 D_{27}D_{48}D_{60}D_{57}^{-1}T_{11}T_{13}^{-1}
 =2^{20/23}3^{1/23}\pi.
\]
The product is one.  This verifies the final gamma simplification.

\section{Cyclotomic verification for the \texorpdfstring{\(D_5\)}{D5} chamber relation}
\label{app:d5-certificate}

Let \(V_{ij}(z)\) be the row-sum polynomial from a chamber in the
\(i\)-th weighted-symmetry orbit to the \(j\)-th orbit, ordered as in
Proposition~\ref{prop:D5-four-orbit-transfer}.  Put
\[
 G(z)=(1+z^2)^3(1+z^4),\qquad H(z)=1+z^2+z^4.
\]
The ten independent entries factor as follows:
\begingroup
\small
\begin{align*}
V_{11}={}&GH(z^{24}-z^{22}+z^{20}-z^{18}+z^{16}-z^{14}+z^{12}
-z^{10}+z^8-z^6+z^4-z^2+1),\\
V_{12}={}&z^4GH(z^4-z^2+1)(z^{12}-z^6+1),\\
V_{13}={}&z^8GH(z^8-z^6+z^4-z^2+1),\\
V_{14}={}&2z^{12}GH,\\
V_{22}={}&G(z^{28}-z^{26}+z^{20}+z^{16}-z^{14}+z^{12}+z^8-z^2+1),\\
V_{23}={}&z^4G(z^{20}-z^{16}+z^{12}+z^{10}+z^8-z^4+1),\\
V_{24}={}&2z^8GH(z^8-z^4+1),\\
V_{33}={}&G(z^8-z^4+1)(z^{20}-2z^{18}+3z^{16}-4z^{14}+5z^{12}
-3z^{10}+5z^8-4z^6+3z^4-2z^2+1),\\
V_{34}={}&2z^4GH(z^8-z^4+1)(z^8-z^6+z^4-z^2+1),\\
V_{44}={}&G(z^8+1)H(z^8-z^4+1)(z^8-z^6+z^4-z^2+1).
\end{align*}
\endgroup
For \(1\le i,j\le3\), one has \(V_{ji}=V_{ij}\).  Since the fourth orbit
has size 96 and the first three have size 48, double counting gives
\(V_{4i}=V_{i4}/2\) for \(i=1,2,3\).  At \(z=1\), the row sums are therefore
exactly the target-orbit sizes.

Let \(r=(1,r_2,r_3,r_4)^T\), with \(r_j\) as in
Proposition~\ref{prop:D5-four-orbit-transfer}.  Exact division gives
\[
 \remd_{\Phi_{19}}(Vr)=0
\]
componentwise.  For the principal \(3\times3\) minor,
\begin{align*}
 \remd_{\Phi_{19}}\det V_{\{1,2,3\},\{1,2,3\}}={}&
 1620z^{17}-512z^{16}+562z^{15}-1302z^{14}-728z^{13}\\
 &-1992z^{12}-1768z^{11}-2214z^{10}-2214z^9-1768z^8\\
 &-1992z^7-728z^6-1302z^5+562z^4-512z^3+1620z^2+2028,
\end{align*}
which is nonzero.  Thus the invariant block has rank exactly three over
\(\Q(\zeta_{19})\).  The same exact replay reconstructs the 83 flats,
\(\chi(T)=(T-1)(T-3)(T-4)(T-5)\), all 240 chambers, and the orbit sizes
\(48,48,48,96\).

\section{Low-rank normalization checks}
\label{sec:low-rank}

This appendix is not used to prove Theorem~\ref{thm:universal-second-pole}.  It records normalization checks against known rank-two and rank-three formulas, followed by several explicit examples obtained from classical beta and Dixon summations.  The table uses the normalized Witten zeta function.
\begin{center}
\renewcommand{\arraystretch}{1.25}
\begin{tabularx}{\textwidth}{@{}c c X@{}}
\toprule
Type & \(q_2\) & \(\displaystyle\sum_i\Per_i(q_2)\)\\
\midrule
\(A_2\)&\(1/2\)&\(2\)\\
\(A_3\)&\(2/5\)&\(\sqrt5\,B(1/5,1/5)\)\\
\(A_4\)&\(1/3\)&\(3\Gamma(1/3)^4/\Gamma(2/3)^2\)\\
\(B_3\)&\(1/4\)&\(\dfrac{2^{1/4}+\csc(\pi/8)}{4\sqrt\pi}
 \Gamma(1/8)\Gamma(3/8)\)\\
\(D_4\)&\(3/11\)&\(3\Per_O+\Per_C\), with \(\Per_O,\Per_C\) below\\
\bottomrule
\end{tabularx}
\end{center}
For \(A_2\), the two wall simplices are points and each period equals
\(1\).  Hence the table entry is independent of the published answer and
\[
 \Res_{s=1/2}\xi_{A_2}(s)
 =\frac{\zeta_{\mathrm R}(1/2)}{2}\sum_i\Per_i(1/2)
 =\frac{\zeta_{\mathrm R}(1/2)}{2}\,2
 =\zeta_{\mathrm R}(1/2),
\]
in exact agreement with Au's Corollary~5.2 and Example~5.3
\cite{Au2024}.  The other rank-two public checks are independent as
well.  In types \(B_2=C_2\) and \(G_2\), the two zero-dimensional face
products are respectively \((2,1)\) and \((2,18)\), so the theorem gives
\[
 \Res_{s=1/3}\xi_{B_2}(s)
 =\frac{\zeta_{\mathrm R}(1/3)}{3}
   \bigl(1+2^{-1/3}\bigr),
 \qquad
 \Res_{s=1/5}\xi_{G_2}(s)
 =\frac{\zeta_{\mathrm R}(1/5)}{5}
   \bigl(2^{-1/5}+18^{-1/5}\bigr),
\]
again exactly as in Au's Corollary~5.2.  The \(A_3\) identity was proved in
\eqref{eq:A3-aggregate}, and its second residue agrees with
\cite[Theorem~9.2]{Au2024}.  Likewise the \(B_3\) entry reproduces
\cite[Theorem~9.4]{Au2024} after division by \(N-1=8\).  Au's
Theorem~9.6 also publishes the \(C_3\) second residue; since an independent
\(C_3\) wall-sum reduction is not included in this manuscript, it is recorded
as prior art rather than used here as a confirmation.

For \(D_4\), there are three outer walls and one central wall.  Direct beta integration
followed by Dixon's \({}_3F_2(1)\) summation gives
\begin{equation}\label{eq:D4-outer}
 \boxed{
 \Per_O=2^{-14/11}\sqrt\pi\,
 \frac{\Gamma(4/11)\Gamma(5/11)\Gamma(8/11)\Gamma(7/22)}
 {\Gamma(19/22)\Gamma(9/11)\Gamma(15/22)},}
\end{equation}
\begin{equation}\label{eq:D4-central}
 \boxed{
 \Per_C=2^{-14/11}\sqrt\pi\,
 \frac{\Gamma(3/22)\Gamma(5/11)\Gamma(7/22)^2}
 {\Gamma(3/11)\Gamma(7/11)\Gamma(9/11)}.}
\end{equation}
Euler reflection yields
\begin{equation}\label{eq:D4-ratio}
 \frac{\Per_C}{\Per_O}
 =\frac{\sin(3\pi/11)\sin(4\pi/11)}
 {\sin(3\pi/22)\sin(7\pi/22)}.
\end{equation}
Therefore
\[
 \Res_{s=3/11}\xi_{D_4}(s)
 =\frac{\zeta_{\mathrm R}(3/11)}{11}(3\Per_O+\Per_C).
\]
The exact reductions in this section are examples; no universal gamma
summation is inferred from them.

\end{document}